\documentclass{amsart}
\usepackage[utf8]{inputenc}
\usepackage{amssymb,amsmath,bm,enumerate,stmaryrd,amsthm,tikz-cd}
\usepackage{stackrel}
\usetikzlibrary{matrix}
\usetikzlibrary{arrows}
\usepackage[all]{xy}
\usepackage[shortlabels]{enumitem}
\usepackage{colonequals}
\makeatletter
\@namedef{subjclassname@2020}{\textup{2020} Mathematics Subject Classification}
\makeatother
\usepackage[colorlinks]{hyperref}
\hypersetup{
 colorlinks=true,
 linkcolor=blue,
 filecolor=magenta, 
 urlcolor=cyan,
}
\usepackage[capitalise]{cleveref}
\crefformat{equation}{(#2#1#3)}
\crefrangeformat{equation}{(#3#1#4--#5#2#6)}
\crefformat{enumi}{(#2#1#3)}
\crefrangeformat{enumi}{(#3#1#4--#5#2#6)}

\usepackage[pagewise]{lineno}
\let\oldequation\equation
\let\oldendequation\endequation
\renewenvironment{equation}{\linenomathNonumbers\oldequation}{\oldendequation\endlinenomath}
\expandafter\let\expandafter\oldequationstar\csname equation*\endcsname
\expandafter\let\expandafter\oldendequationstar\csname endequation*\endcsname
\renewenvironment{equation*}{\linenomathNonumbers\oldequationstar}{\oldendequationstar\endlinenomath}
\let\oldalign\align
\let\oldendalign\endalign
\renewenvironment{align}{\linenomathNonumbers\oldalign}{\oldendalign\endlinenomath}
\expandafter\let\expandafter\oldalignstar\csname align*\endcsname
\expandafter\let\expandafter\oldendalignstar\csname endalign*\endcsname
\renewenvironment{align*}{\linenomathNonumbers\oldalignstar}{\oldendalignstar\endlinenomath}

\usepackage{tikz-cd}
\usetikzlibrary{decorations.pathmorphing}

\makeatletter
\newcommand{\xra}{\xrightarrow}

\newcommand{\del}{\partial}

\newcommand{\dcat}[2][]{\cat{D}_{#1}(#2)}

\newcommand{\dfcat}[2][]{\cat{D}^{\sf{f}}_{#1}(#2)}
\newcommand{\dfbcat}[1]{\dfcat[\sf{b}]{#1}}
\newcommand{\dfpcat}[1]{\dfcat[+]{#1}}
\newcommand{\dfmcat}[1]{\dfcat[-]{#1}}

\newcommand{\drel}[2]{\cat{D_b}({#1}/{#2})} 
\newcommand{\perf}{\mathsf{Perf}}

\newcommand{\cob}[1]{\mathcal{#1}}

 \newcommand{\mcV}{\mathcal{V}}

\newcommand{\fm}{\mathfrak{m}}

\newcommand{\fp}{\mathfrak{p}}

\newcommand{\bmf}{\bm{f}}

\DeclareMathOperator{\V}{V}
\DeclareMathOperator{\ann}{ann}

\newcommand{\cat}[1]{{\mathsf{#1}}}

\newcommand{\env}{\mathsf{e}}
\DeclareMathOperator{\Ext}{Ext}

\newcommand{\builds}[1][]{\mathrel{\mathop{\models}\limits_{\vbox to 0.5em{\kern-3\ex@\hbox{$\scriptstyle\,\,#1$}\vss}}}}
\newcommand{\h}{\operatorname{H}}
\newcommand{\bgg}{\mathsf{d}}
\newcommand{\bg}{\mathsf{b}}
\newcommand{\univ}{\mathsf{u}}

\DeclareMathOperator{\Hom}{Hom}

\DeclareMathOperator{\Image}{Im}
\newcommand{\susp}{{\mathtt{\Sigma}}}

\newcommand{\lotimes}{\otimes^{\operatorname{L}}}

\newcommand{\proj}{\operatorname{Proj}}
\newcommand{\join}{\operatorname{Join}}

\newcommand{\RHom}{\operatorname{RHom}}
\newcommand{\shift}{{\sf\Sigma}}

 \newcommand{\Sym}{\operatorname{Sym}}
\newcommand{\supp}{\operatorname{Supp}}
\newcommand{\ssupp}{\operatorname{supp}}

\DeclareMathOperator{\Tor}{Tor}

\newcommand{\coloneq}{\colonequals}
\makeatother

\newtheorem*{introthm}{Theorem}

\newtheorem{theorem}[subsection]{Theorem}
\newtheorem{proposition}[subsection]{Proposition}
\newtheorem{lemma}[subsection]{Lemma}
\newtheorem{corollary}[subsection]{Corollary}

\theoremstyle{definition}

\newtheorem{chunk}[subsection]{}

\theoremstyle{remark}
\newtheorem{remark}[subsection]{Remark}

\newtheorem*{ack}{Acknowledgements}

\newtheorem{notation}[subsection]{Notation}

\numberwithin{equation}{subsection}

\title[Tensor products]{Cohomological supports of tensor products of \\ modules over commutative rings}

\author[S.B.~Iyengar]{Srikanth~B.~Iyengar}
\address{Srikanth~B.~Iyengar,
Department of Mathematics,
University of Utah,
Salt Lake City, UT 84112,
U.S.A.}
\email{iyengar@math.utah.edu}

\author[J.~Pollitz]{Josh Pollitz}
\address{Josh Pollitz,
Department of Mathematics,
University of Utah,
Salt Lake City, UT 84112,
U.S.A.}
\email{pollitz@math.utah.edu}

\author[W.~T.~Sanders]{William~T.~Sanders}
\address{William~T.~Sanders,
IOTA Foundation,
Pappelallee 78/79,
10437 Berlin, Germany}
\email{william.sanders@iota.org}

\date{13 March 2022}



\begin{document}

\begin{abstract}
This works concerns cohomological support varieties of modules over commutative  local rings. The main result is that the  support of a derived tensor product of a pair of differential graded modules over a Koszul complex is the join  of  the supports of the modules. This generalizes, and gives another proof of, a result of Dao and the third author dealing with Tor-independent modules over complete intersection rings. The result for Koszul complexes has a broader applicability, including to exterior algebras over local rings.
\end{abstract}

\maketitle

\section*{Introduction}

 Throughout we fix a Koszul complex $E$ over a (commutative noetherian) local ring $(R,\fm,k)$  on a list of elements $\bmf=f_1,\ldots,f_c$ in $\fm. $  As explained in \cite{Pollitz:2021} studying the homological properties of differential graded (abbreviated to dg) $E$-modules  allows one to unify and extend the results about quotients $R\to R/(\bmf)$ when $\bmf$ is an $R$-regular sequence as well as those about exterior algebras over $R$. The dg $E$-modules perfect when regarded as $R$-complexes---in the sense that they are quasi-isomorphic to a bounded complex of finite rank free $R$-modules---are the ones that exhibit especially structured homological phenomena; see, for example, \cite{Avramov:1989,Avramov/Buchweitz:2000b,Avramov/Gasharov/Peeva:1997,Avramov/Sun:1998,Briggs/McCormick/Pollitz:2022,Eisenbud:1980,Gulliksen:1974,Stevenson:2014}. The homological properties of such a dg $E$-module $M$ are often encoded in its cohomological support, denoted $\V_E(M)$, which is a naturally associated Zariski closed subset of $\mathbb{P}_k^{c-1};$ cf.\@ \cref{c:univ2}.

The main result of this article is the following. 
 \begin{introthm}
 For dg $E$-modules $M,N$ that are perfect over $R$ there is an equality
 \[
 \V_E(M\lotimes_E N)=\join(\V_E(M),\V_E(N))\,.
 \]
 \end{introthm}
 
The join of closed subsets $U,V$ of $\mathbb{P}_k^{c-1}$, denoted $\join(U,V)$, is the closure of the union  of lines connecting a point from $U$ to a point from $V$; see \cref{c:join} for details. Specializing the theorem above to the case where $R$ is a regular ring, $\bmf$ is an $R$-regular sequence, and $M,N$ are finitely generated $R/(\bmf)$-modules satisfying $\Tor^E_i(M,N)=0$ for all $i\ge 1$ recovers \cite[Theorem~3.1]{Dao/Sanders:2017}.  The proof in \emph{loc.\@ cit.\@} involves a series reductions and ad hoc geometric arguments. Besides  generalizing this result, a main point of this article is to offer a simpler proof by a passage to an exterior algebra, as briefly described below. 

The theorem above is proved in Section~\ref{s:tp-II}.  As a corollary, we deduce that when R is Gorenstein and  $\RHom_E(M,N)$ is perfect as an $R$-complex, there is an equality
\[
\V_E(\RHom_E(M,N)) = \join(\V_E(M),\V_E(N))\,.
\]
This is Corollary~\ref{cor:rhom} and it generalizes \cite[Theorem~4.7]{Dao/Sanders:2017}. Theorem~\ref{t:second} relates the support of the dg module $M\lotimes_EN$ to that of its homology modules, namely, $\Tor^E_i(M,N)$, thereby providing a positive answer to \cite[Question~2]{Dao/Sanders:2017}.

The key ingredient in our work is a functor, denoted $\mathsf{t}$, from the derived category of dg $E$-modules $\dcat{E}$ to the derived category of dg $\Lambda$-modules $\dcat{\Lambda}$ where $\Lambda$ is an exterior algebra on $\shift k^c$; see \cref{s:passage}. The relevance of this functor arises from \cref{l:passage} which identifies $\V_E(M)$ with $\V_\Lambda(\mathsf{t} M)$, and that as dg $\Lambda$-modules 
\[
\mathsf{t}(M\lotimes_E N)\simeq \mathsf{t}M \lotimes_\Lambda \mathsf{t} N\,.
\]
The expression for $\V_E(M\lotimes_EN)$ in the theorem above is a consequence of \cref{p:derivedtensor} that asserts if $X,Y$ are dg $\Lambda$-modules with finite dimensional homology, then
\[
\V_\Lambda(X\lotimes_\Lambda Y)=\join(\V_\Lambda(X),\V_\Lambda(Y))\,. \tag{$\dagger$}
\]
This equality is in turn deduced using a contravariant version, from \cite{Avramov:2013}, of the Bernstein-Gelfand-Gelfand correspondence functor:
\[
\bgg\colon \dcat{\Lambda}\to\dcat{\cob S}\,,
\]
where $\cob S$ is the symmetric algebra on $\shift^{-2}k^c$. The main calculation in the proof of ($\dagger$) is the interaction between tensor products and the functor $\bgg$, namely: Given dg $\Lambda$-modules $X,Y$ with homology finite dimensional over $k$,  there is an isomorphism of dg $\cob S$-modules
\[
\bgg (X\lotimes_\Lambda Y)\simeq \bgg X\otimes_k \bgg Y\,,
\] 
where the right-hand side is regarded as a dg $\cob S$-module through a natural map of $k$-algebras $\Delta\colon \cob S\to \cob S\otimes_k \cob S$, which makes $\cob S$ into a Hopf algebra; see \cref{polynomialdiagonal}. Given this result, ($\dagger$) follows by a standard argument concerning supports of modules over polynomial rings, discussed in \cref{s:support}; see especially \cref{l:support}. The isomorphism above, which is folklore, is contained in \cref{p:derivedtensor}.

\begin{ack}
It is our pleasure to thank the referees for their comments and questions on an earlier version of this article. This work was partly supported by National Science Foundation DMS grants 2001368 (SBI) and 2002173 (JP).
\end{ack}

\section{Joins and supports}
\label{s:support}
Let $k$ be a field. In what follows we encounter graded $k$-vector spaces $W$ whose natural grading is lower and also those whose natural grading is upper. It is convenient to adopt the convention that $W$ has both an upper and a lower grading, with $W^i=W_{-i}$ for each integer $i$. We indicate the primary grading when necessary.

Fix a finite dimensional graded $k$-space $W\coloneq \{W^i\}_{i\in \mathbb{Z}}$ concentrated in positive even degrees. Let $\cob S\coloneq \Sym_k W$ be the symmetric algebra (over $k$) on $W$ and $\proj \cob S$ the set of homogeneous prime ideals of $\cob S$ not containing the irrelevant maximal ideal $\cob S^{>0}$ of $\cob S$, equipped with the Zariski topology. In this section we recall some basics on joins of closed subsets of $\proj{\cob S}$ and of supports of graded $\cob S$-modules. Our standard references are \cite[Section 1.3]{Flenner/O'Carroll/Vogel:1999}, for joins, and \cite{Foxby:1979}, for supports.

\begin{chunk}
\label{c:join}
The map $W\to W\oplus W$ given by $w\mapsto (w, 0)+(0,w)$ induces a map 
\begin{equation}
\label{polynomialdiagonal}
\Delta\colon \cob S\to \cob S\otimes_k \cob S
\end{equation} 
of graded $k$-algebras and makes $\cob S$ into a graded Hopf algebra over $k$. It also defines a rational map
\[
\delta\colon \proj (\cob S\otimes_k \cob S)\dashrightarrow\proj \cob S
\] 
that is defined (and regular) off of the anti-diagonal $D$ in $\proj (\cob S\otimes_k \cob S)$; here $D$ is the image of the embedding $\proj {\cob S}\hookrightarrow \proj (\cob S\otimes_k \cob S)$ determined by the map $W\oplus W\to W$ given by $(w_1,w_2)\mapsto w_1+w_2$.
	
Given closed subsets $U\coloneq \cob V(\cob I)$ and $V\coloneq \mcV(\cob J)$ of $\proj \cob S$,  consider 
\[
	J(U,V)\coloneq\proj(\cob S /\cob I\otimes_k \cob S /\mathcal{J})
\] 
viewed as a closed subset of $\proj(\cob S\otimes_k \cob S)$. The \emph{join} of $U$ and $V$, denoted $\join(U,V)$, is the closure in $\proj \cob S$ of the set	
\[
\delta\left( J(U,V)\smallsetminus D\right)\,.
\] 
When $k$ is algebraically closed, the Nullstellensatz identifies $\proj \cob S$ with projective space $\mathbb{P}^{d-1}_k$ where $d=\dim_k W.$ Under this identification, the join of $U$ and $V$ is the closure of the union of lines in $\proj \cob S$ containing a point $u$ in $U$ and a point $v$ in $V$.
\end{chunk}

\begin{remark}
The join can also be defined as follows: Consider the rational map 
\[
\delta'\colon \proj (\cob S\otimes_k \cob S)\dashrightarrow \proj \cob S\,,
\] that is regular of the diagonal in $\proj (\cob S\otimes_k \cob S)$, induced by the $k$-algebra map $\cob S\to \cob S\otimes_k \cob S$
determined by $w\mapsto w\otimes 1-1\otimes w.$ The linear automorphism $\alpha$ of  $\proj (\cob S\otimes_k \cob S)$ determined by 
\[
w\otimes 1\mapsto w\otimes 1\quad \text{and}\quad 1\otimes w\mapsto -1\otimes w
\]fixes $J(U,V)$ for any pair of closed subsets $U,V$ of $\proj \cob \cob S$, maps $D$ bijectively to $D'$ and $\delta=\delta'\alpha$.
Hence
\[
\delta(J(U,V)\smallsetminus D)=\delta'(J(U,V)\smallsetminus D')
\]
where the right-hand side is the definition of the join used in \cite[Section~1.3]{Flenner/O'Carroll/Vogel:1999}. That is the definitions of joins from \emph{loc.\@ cit.\@} and \cref{c:join} coincide. We opt for the latter as the isomorphism in \cref{p:derivedtensor} respects $\Delta$.
\end{remark}

\begin{chunk}
\label{c:support}
Let $X$ be a graded $\cob S$-module. The \emph{support of $X$} over $\cob S$ is the subset
\[
\supp_\cob S X\coloneq \{\fp\in \proj \cob S\mid X_\fp\neq 0\}\,,
\] 
where $X_\fp$ denotes the homogeneous localization of $X$ at $\fp$. Following Foxby~\cite{Foxby:1979}, the \emph{small support of $X$} is 
\[
\ssupp_\cob S X\coloneq \{\fp\in \proj \cob S\mid X\lotimes_S \kappa(\fp)\neq 0\}\,,
\]
where $\kappa(\fp)$ is the graded field $\cob S_\fp/\fp \cob S_\fp$. Consider the closed subset 
\[
\mcV(\ann_\cob S X) \coloneq \{\fp\in \proj \cob S\mid \fp\supseteq \ann_\cob S X\}
\]
of $\proj \cob S$. In general there are inclusions
\begin{equation}\label{e:containment}
\ssupp_\cob S X\subseteq \supp_\cob S X\subseteq \mcV(\ann_\cob S X)\,.
\end{equation}
Moreover, $\supp_\cob S X$ is the specialization closure of $\ssupp_\cob S X;$ see \cite[Lemma~2.2]{Benson/Iyengar/Krause:2008}.
Equalities hold when the $\cob S$-module $X$ is finitely generated.
\end{chunk}

\begin{lemma}
\label{l:support}
Let $X,Y$ be finitely generated graded $\cob S$-modules.	There is an equality 
\[
\supp_\cob S (X\otimes_k Y)=\join(\supp_\cob S X,\supp_\cob S Y)
\] 
where $X\otimes_k Y$ is regarded as a graded $\cob S$-module via \cref{polynomialdiagonal}. 
\end{lemma}

\begin{proof}
As a matter of notation, we write  $\cob S^\env$ for $\cob S\otimes_k \cob S$ and  use $\overline{(-)}$ for closure in the Zariski topology. 
For any finite generated $\cob S^\env$-module $N$ and $\fp\in \proj \cob S $ one has
\[
N\lotimes_{\cob S} \kappa(\fp)\simeq N\lotimes_{\cob S^\env} (\cob S^\env\lotimes_{\cob S} \kappa(\fp))\,.
\]  
This leads to the following equivalences:
\begin{align*}
    \fp\in \ssupp_\cob S N&\iff \ssupp_{{\cob S^\env}}(N)\cap \ssupp_{\cob S^\env}({\cob S^\env}\lotimes_{\cob S} \kappa(\fp))\neq\emptyset\\
    &\iff \ssupp_{{\cob S^\env}}(N)\cap(\delta^{-1}(\fp)\smallsetminus D) \neq \emptyset\\
    &\iff \fp\in \delta(\ssupp_{\cob S^\env} (N)\smallsetminus D)\,.
\end{align*}
Applying this observation to $N\colonequals X\otimes_kY$ justifies the last equality below:
\begin{align*}
 \join(\supp_\cob S X, \supp_\cob S Y)&=\overline{\delta(\supp_{\cob S^\env} (X\otimes_k Y)\smallsetminus D)}\\
 &=\overline{\delta(\ssupp_{\cob S^\env} (X\otimes_k Y)\smallsetminus D)}\\
 &=\overline{\ssupp_\cob S (X\otimes_k Y)}\,.
 \end{align*}
The first equality holds as $X,Y$ are finitely generated over $\cob S$, whilst the second equality holds because $X\otimes_kY$ is finitely generated over $\cob S^\env$. Thus, for the desired statement, it suffices to verify that
 \[
 \overline{\ssupp_\cob S (X\otimes_k Y)}=\supp_\cob S (X\otimes_k Y)\,.
 \] 
To that end, given \cref{c:support}, it suffices to verify that $\supp_\cob S (X\otimes_k Y)$ is closed in $\proj\cob S$. As an $\cob S$-module $X\otimes_k Y$ need not be finitely generated, but it is finitely generated over $\cob S^\env$, and that suffices. 
 
Indeed, let $G$ be a finite generating set for $X\otimes_k Y$ over $\cob S^\env$ and $T$ the $\cob S$-submodule of $X\otimes_kY$ generated by $G$; here $\cob S$ is acts via the diagonal map \cref{polynomialdiagonal}. 
Since $T$ is finitely generated over $\cob S$, one gets the first equality below:
 \begin{align*}
\mcV(\ann_\cob S T)&=\supp_\cob S T\\
&\subseteq\supp_\cob S (X\otimes_k Y) \\
&\subseteq\mcV(\ann_\cob S (X\otimes_k Y)) \\
&= \mcV(\ann_\cob S T)\,.
 \end{align*} 
The containments are from \cref{c:support}; the last equality holds as $\ann_\cob S (X\otimes_k Y)=\ann_\cob S T$. Thus the inclusions above are equalities, as desired.
 \end{proof}

\section{Dg modules over graded algebras }
Let $A=\{A_i\}_{i\in \mathbb{Z}}$ be a strictly graded-commutative dg algebra. Its homology algebra, $\h(A)$, is thus also strictly graded-commutative. 

\begin{chunk}
A dg $A$-module $F$ is \emph{semifree} provided it admits an exhaustive filtration 
\[
0=F(-1)\subseteq F(0)\subseteq F(1)\subseteq \ldots \subseteq F
\] where each subquotient $F(i)/F(i-1)$ is a coproduct of suspensions of $A$. A \emph{semifree} resolution of a dg $A$-module $M$ is a surjective quasi-isomorphism of dg A-modules $F\xra{\simeq} M$ where $F$ is a semifree dg A-module. Such resolutions of $M$ exist and any two are unique up to homotopy equivalence; see, for example, \cite[6.6]{Felix/Halerpin/Thomas:2001}. 
\end{chunk}

\begin{chunk}Let $M$ be a dg $A$-module and fix $F\xra{\simeq} M$ a semifree resolution over $A$. By \cite[ ]{Felix/Halerpin/Thomas:2001}, the functors $F\otimes_A -$ and $\Hom_A(F,-)$ preserve (surjective) quasi-isomorphisms. Hence by replacing objects with their semifree resolutions we obtain bi-functors $-\lotimes_A-$ and $\RHom_A(-,-)$ on $\dcat{A}$; that is to say, 
\[
M\lotimes_A -\coloneq F\otimes_A - \quad \text{ and }\quad \RHom_A(M,-) \coloneq \Hom_A(F,-)\,.
\]
As usual, we set
\[
 \Tor^A_*(M,N) \coloneq \h_*(M\lotimes_AN) \quad\text{ and }\quad \Ext_A^*(M,N)\coloneq \h^*(\RHom_A(M,N))\,.
 \]
 As $A$ is graded-commutative these are graded $\h(A)$-modules. 
\end{chunk}

\begin{chunk}The derived category of dg $A$-modules is denoted $\dcat A$, and it is regarded as a triangulated category in the standard way; see, for example, \cite[Section~2]{Avramov/Buchweitz/Iyengar/Miller:2010}.  The suspension functor associates to each dg $A$-module $M$ the dg module $\shift M$ with
\[
(\shift M)_i=M_{i-1}, \quad \del^{\shift M}=-\del^M \quad \text{ and }\quad a\cdot \susp m=(-1)^{|a|}am\,
\]
where $|a|$ denotes the  degree $a$. A \emph{thick} subcategory of a triangulated category is a triangulated subcategory that is closed under retracts. 
\end{chunk}

\begin{chunk}
Let $A$ be a dg algebra over a field $k$. 
We define several thick subcategories of $\dcat{A}$ that will be of interest in what follows.

Let $\dfpcat{A}$ denote the full subcategory of $\dcat{A}$ consisting of dg $A$-modules $M$ with each $\h_i(M)$ finite dimensional and $\h_{i}(M)=0$ for all $i\ll 0;$ define $\dfmcat{A}$ analogously where the second condition is replaced with $\h_{i}(M)=0$ for all $i\gg 0$. We let $\dfbcat{A}$ denote $\dfpcat{A}\cap \dfmcat{A}.$ That is, $\dfbcat{A}$ consists exactly of those dg $A$-modules whose homology is finite dimensional over $k$. We write $\perf(A)$ for the thick subcategory of $\dcat{A}$ generated by $A$; see \cite[Theorem~4.2]{Avramov/Buchweitz/Iyengar/Miller:2010} for an alternative characterization.
\end{chunk}

\section{Exterior algebras}
In this section $V\coloneq \{V_i\}_{i\in \mathbb{Z}}$ is a finite graded $k$-space concentrated in positive odd degrees. Set $(-)^\vee\coloneq \Hom_k(-,k)$, the graded dual, and $W\coloneq  \shift^{-1}(V^{\vee})$.
Let 
\[
\Lambda\coloneq \bigwedge_k V  \quad \text{and} \quad \cob S\coloneq \Sym_k W\,;
\]
the former being the exterior algebra, over $k$, on $V$. Set $\Gamma\coloneq {\cob S}^\vee$ with the standard $\cob S$-module structure: For $\alpha\in \Gamma$ and $\chi\in \cob S$, one has
 \[
 \chi\cdot \alpha\coloneq \alpha(\chi\cdot -)\,.
 \]
 We view $\Lambda$ as a graded Hopf algebra, with coproduct $\Lambda \to \Lambda\otimes_k \Lambda$ induced by the map of $k$-spaces $v\mapsto (v,1) + (1,v)$, for $v\in V$. Hence for any left dg $\Lambda$-module the antipode defines a dg $\Lambda$-module structure on $M^\vee$. Also, for a pair of dg $\Lambda$-modules $M,N$, their tensor product $M\otimes_k N$ is regarded as a dg $\Lambda$-module through the coproduct. See, for example, \cite[Remark~5.2]{Avramov/Iyengar:2010}. We also view $\cob S$ as a graded Hopf algebra over $k$, with coproduct defined in \cref{c:join}.

\begin{notation}
Fix a basis $e_1,\dots,e_c$ for $V$, and let $\chi_1,\ldots,\chi_c$ be the dual basis for $W$; thus $\chi_i$ has lower degree $-|e_i|-1$.    These determine isomorphisms 
\[
\Lambda \cong \bigwedge\left(ke_1\oplus \ldots \oplus ke_c\right) \quad \text{ and } \quad  \cob S\cong k[\chi_1,\ldots,\chi_c] \,.
\] 
\end{notation}

\begin{chunk}
\label{c:universalresolutions}
For a dg $\Lambda$-module $M$, its \emph{universal resolution} $\univ M$ is the dg $(\Lambda\otimes_k \cob S)$-module with underlying graded $(\Lambda\otimes_k \cob S)$-module $\Lambda\otimes_k \Gamma \otimes_k M$, with $\Lambda\otimes_k \cob S$ acting by left multiplication on the two left factors, and differential 
\[ 
1\otimes 1\otimes\del^M+\sum_{i=1}^c (1\otimes\chi_i\otimes e_i-e_i\otimes\chi_i\otimes 1)\,.
\] 
The canonical projection $\univ M\to M$ is a semifree resolution of $M$ over $\Lambda$; see \cite[Proposition~2.6]{Avramov/Buchweitz:2000a} or \cite[Section~7]{Avramov/Buchweitz/Iyengar/Miller:2010}. Moreover, since $\univ M$ is a dg module over $\Lambda\otimes_k\cob S$, the graded $k$-space $\Hom_\Lambda (\univ M,-)$ retains a dg $\cob S$-module structure and so
\[
\Ext_\Lambda^*(M,-)=\h^*(\Hom_\Lambda (\univ M,-))
\] 
is a graded $\cob S$-module. 
\end{chunk}

\begin{chunk}
\label{c:bgg}
Let $M$ be dg $\Lambda$-module with $M_i$ degreewise finite dimensional over $k$ for each $i$ and $0$ for $i\ll 0$; up to quasi-isomorphism any object in $\dfpcat{\Lambda}$ has this form. There is an isomorphism of dg $\cob S$-modules
\begin{equation}
\label{e:curveddg}
\Hom_\Lambda (\univ M,k)\cong\cob S\otimes_k M^\vee 
\end{equation} 
where the right-hand term has differential $1\otimes \del^{M^\vee }+\sum_{i=1}^c \chi_i\otimes e_i;$ we denote the dg $\cob S$-module on the right by $\cob S_{M}$. From this isomorphism and \cite[Proposition 1.2.8]{Pollitz:2021}, $\Hom_\Lambda (\univ M,k)$ is a semifree dg $\cob S$-module. 

The contravariant functor $\Hom_\Lambda(\univ(-),k)$ induces the exact functor 
\[
\bgg\colon \dcat{\Lambda}^{\textsf{op}}\to \dcat{\cob S}\,.
\] 
By \cite{Avramov:2013}, this restricts to an exact equivalence 
\[
\bgg\colon \dfpcat{\Lambda}^{\textsf{op}}\xra{\equiv} \dfmcat{\cob S}\,,
\] 
that further restricts to equivalences 
\[
\dfbcat{\Lambda}^{\textsf{op}}\xra{\equiv} \perf (\cob S)\quad \text{ and }
	\quad \perf(\Lambda)^{\textsf{op}}\xra{\equiv} \dfbcat{\cob S}\,.
\]
One has also the functor $\Hom_\Lambda(\univ k,-)$ that induces an exact functor 
\[
\bg \colon \dcat{\Lambda}\to \dcat{\cob S}\,
\]  which restricts to equivalences 
\[
\dfbcat{\Lambda}\xra{\equiv} \perf (\cob S)\quad \text{ and }
	\quad \perf(\Lambda)\xra{\equiv} \dfbcat{\cob S}\,.
\]
cf.\@ \cite{Avramov/Buchweitz/Iyengar/Miller:2010}. There is the following commutative diagram
\[
 \begin{tikzcd}
 \dfpcat{\Lambda}^{\textsf{op}} \ar[d,swap, "(-)^\vee "] \ar[r,"\bgg"]& \dfmcat{\cob S}\,. \\
 \dfmcat{\Lambda} \ar[ru,swap,"\bg"]& 
 \end{tikzcd}
\]
\end{chunk}

\begin{chunk}
The functors $\bg,\bgg$ defined above determine two notions of cohomological support for dg $\Lambda$-modules. Namely, for a dg $\Lambda$-module $M$, consider subsets of $\proj \cob S$
\begin{align*}
 \V_\Lambda^\bg(M) &\colonequals \supp_\cob S \h (\bg M )=\supp_\cob S \Ext_\Lambda(k,M)\\
 \V_\Lambda^\bgg(M)&\colonequals \supp_\cob S \h (\bgg M )=\supp_\cob S \Ext_\Lambda(M,k)\,.
\end{align*}
\end{chunk}
In \cite{Carlson/Iyengar:2015}, the supports $\V_\Lambda^\bg(-)$ are used to classify the thick subcategories of $\dfbcat{\Lambda}$. 
Our focus will be on $\V_\Lambda^\bgg(-)$ but it is worth recording their relationship.

\begin{proposition}
Let $M$ be in $\dfpcat{\Lambda}$. There is an equality $\V_\Lambda^\bgg(M)=\V_\Lambda^\bg(M^\vee )$.
Moreover if $M$ is in $\dfbcat{\Lambda}$, then  $\V_\Lambda^\bgg(M)=\V_\Lambda^\bg(M)$.
\end{proposition}

\begin{proof}
The first equality is immediate from $\bg((-)^\vee )=\bgg$; see \cref{c:bgg}. The second equality follows from the first. Indeed, it is easy to check that if $N$ is in the thick subcategory generated by $N'$ then
\[
\V_\Lambda^\bg(N)\subseteq \V_\Lambda^\bg(N')\quad \text{ and }\quad \V_\Lambda^\bgg(N)\subseteq \V_\Lambda^\bgg(N')\,.
\]
When $M$ is in $\dfbcat{\Lambda}$, the dg $\Lambda$-modules $M$ and $M^\vee $ generate the same thick subcategory---see \cite[Section~4]{Liu/Pollitz:2021}---so the second equality follows from the first. 
\end{proof}

\section{Support for tensor products, I}
As in the previous section, $V\colonequals \{V_i\}_{i>0}$ is a finite graded $k$-space concentrated in positive even degrees, and 
\[
\Lambda\coloneq \bigwedge V \quad \text{and} \quad  \cob S \coloneq  \Sym_k W\,,
\]
where $W\colonequals \shift^{-1}(V^{\vee})$. In this section we analyze the interaction between the functor $\bgg\colon \dfpcat{\Lambda}^{\text{op}}\to \dfmcat{\cob S} $, from \cref{c:bgg}, and the tensor products $\otimes_k$ and $\lotimes_\Lambda$. The main results, \cref{p:hopftensor} and \cref{p:derivedtensor}, are folklore but we could not find adequate references, so we give complete proofs; see also \cref{r:buchweitz}.

\begin{lemma}
\label{l:NAK}
For ${\cob S}$-modules $X,Y$ with finitely generated homology, 
\[
\supp_{\cob S}\h(X\lotimes_S Y)=\supp_{\cob S} \h(X)\cap \supp_{\cob S}\h(Y)\,.
\]
\end{lemma}

\begin{proof}
Since $X,Y$ have finitely generated homology, there are equalities
 \begin{align*}
     \supp_{\cob S} \h(X)&=\{\fp\in \proj \cob S\mid X\lotimes_{\cob S}\kappa(\fp)\not\simeq 0\}\\
      \supp_{\cob S} \h(Y)&=\{\fp\in \proj \cob S\mid Y\lotimes_{\cob S}\kappa(\fp)\not\simeq 0\}\,.
 \end{align*} 
 See, for instance, \cite[Theorem~2.4]{Carlson/Iyengar:2015}.  Since $\cob S$ has finite global dimension, the $\cob S$-module $\h(X\lotimes_{\cob S} Y)$ is also finitely generated and so
 \[
  \supp_{\cob S} \h(X\lotimes_{\cob S} Y)=\{\fp\in \proj {\cob S}\mid X\lotimes_{\cob S} Y\lotimes_{\cob S}\kappa(\fp)\not\simeq 0\}\,.
 \] 
The desired equality follows from the ones above and the isomorphism
 \[
 X\lotimes_{\cob S} Y\lotimes_{\cob S}\kappa(\fp)\simeq (X\lotimes_{\cob S} \kappa(\fp))\otimes_{\kappa(\fp)}(Y\lotimes_{\cob S} \kappa(\fp))\,. \qedhere
 \] 
\end{proof}

The result below records the relationship between $\bgg$ and tensor products.

\begin{proposition}
\label{p:hopftensor}
For $M,N$ in $\dfpcat{\Lambda}$, there is an isomorphism of dg $\cob S$-modules 
\[
\bgg(M\otimes_k N)\simeq \bgg M\lotimes_{\cob S} \bgg N\,.
\] 
Furthermore if $M,N$ are in $\dfbcat{\Lambda}$, then 
\[
\V_\Lambda^\bgg(M\otimes_k N)=\V_\Lambda^\bgg(M)\cap \V_\Lambda^\bgg(N)\,.
\]
\end{proposition}

\begin{proof}
Replacing $M$ and $N$ with semifree resolutions over $\Lambda$, we may assume both $M$ and $N$ are bounded below and degreewise finite dimensional over $k$, as in \cref{c:bgg}. Let $\Phi$ denote the composition of the isomorphisms of dg $\cob S$-modules
\[
 (\cob S\otimes_k M^\vee )\otimes_{\cob S} (\cob S\otimes_k N^\vee ) \longrightarrow
 (\cob S\otimes_{\cob S} \cob S)\otimes_k M^\vee \otimes_k N^\vee \longrightarrow \cob S\otimes_k M^\vee \otimes_k N^\vee 
\]
where the first one is the twist isomorphism given by
\[
(s\otimes \alpha)\otimes (s'\otimes \beta)\mapsto (s\otimes s')\otimes(\alpha\otimes\beta)
\]
and the second map is the multiplication isomorphism. It is straightforward to see 
\[
\Phi \circ \sum_{i=1}^c(\chi_i\otimes e_i)\otimes 1+1\otimes (\chi_i\otimes e_i) =\sum_{i=1}^c\chi_i\otimes (e_i \otimes 1+ 1\otimes e_i)\circ \Phi\,.
\]
As $M,N$ are degreewise finite dimensional and bounded below, there is a natural isomorphism of dg $\Lambda$-modules
\[
(M\otimes_k N)^\vee \cong M^\vee \otimes_k N^\vee\,.
\]
Hence $\Phi$ yields an isomorphism 
\[
{\cob S}_{M}\otimes_{\cob S}{\cob S}_{N}\xra{\cong} {\cob S}_{M\otimes_k N}\,.
\] 
As a consequence, \cref{e:curveddg} establishes the isomorphisms in $\dcat{\cob S}$:
\begin{equation}
\label{e:tensoriso}
\bgg(M\otimes_k N)\simeq \bgg M \otimes_{\cob S} \bgg N \simeq \bgg M\lotimes_{\cob S} \bgg N\,.
\end{equation}

As for the statement regarding supports, consider the following equalities 
\begin{align*}
\V_\Lambda^\bgg(M\otimes_k N)&=\supp_{\cob S} \h(\bgg(M\otimes_k N)) \\
&=\supp_{\cob S} \h(\bgg M\lotimes_{\cob S} \bgg N) \\
&=\supp_{\cob S} \h(\bgg M )\cap \supp_{\cob S}\h^*( \bgg N )\\
&=\V_\Lambda^\bgg(M)\cap \V_\Lambda^\bgg(N)\,.
\end{align*}
 The second equality is from \cref{e:tensoriso}, while the third is \cref{l:NAK}. 
\end{proof}

\begin{remark}
\label{r:buchweitz}
Buchweitz proved that if $M,N$ are  graded $\Lambda$-modules that are bounded below and are degreewise finite rank over $k$, then
\begin{equation}
\label{e:buchweitz}
 \bg (M\otimes_k N)\simeq \bg M\lotimes_{\cob S} \bg N\,;
\end{equation} 
see \cite[(9.4.10)]{Buchweitz:2021}.  It is easy to see that this isomorphism holds for all pairs of objects in $\dfpcat{\Lambda}$. From the equality $\bg((-)^\vee)=\bgg$ the isomorphisms in \cref{e:buchweitz} can also be deduced from (and imply) the ones in \cref{p:hopftensor}.
\end{remark}

\begin{proposition}
\label{p:derivedtensor}
For $M,N$ in $\dfpcat{\Lambda}$ there is an isomorphism of dg $\cob S$-modules
\[
\bgg(M\lotimes_\Lambda N)\simeq \bgg M\otimes_k \bgg N
\]
where the right hand side is a dg $\cob S$-module through the diagonal $\Delta\colon \cob S\to \cob S\otimes_k \cob S$ described in \cref{polynomialdiagonal}. Furthermore if $M,N$ are in $\dfbcat{\Lambda}$, then 
\[
\V_\Lambda^\bgg(M\lotimes_\Lambda N)=\join(\V_\Lambda^\bgg(M), \V_\Lambda^\bgg(N))\,.
\]
\end{proposition}

\begin{proof}
Replacing $M$ and $N$ with suitable resolutions we can assume $M$ and $N$ are bounded above, degreewise finite dimensional over $k$, and  semifree. As in \cref{c:universalresolutions} we consider $\Gamma={\cob S}^\vee $, regarded as a graded $\cob S$-module. Forgetting differentials one has a commutative diagram 
\[
 \begin{tikzcd}
 \univ M\otimes_\Lambda \univ N\ar[r,dotted,"\Phi"] \ar[d,swap,"\cong"] & \univ (M\otimes_\Lambda N) \\
 \Lambda\otimes_k \Gamma^{\otimes 2} \otimes_k M\otimes_k N \ar[r,"1\otimes \mu\otimes \pi"]& \Lambda\otimes_k \Gamma\otimes_k (M\otimes_\Lambda N) \ar[u,"\cong"]
 \end{tikzcd}
\]
of graded $\cob S$-modules, where the map on the bottom is defined using the  multiplication $\mu\colon \Gamma\otimes_k \Gamma \to \Gamma$ map which is dual to the diagonal $\Delta\colon \cob S\to \cob S\otimes_k \cob S$ in \cref{polynomialdiagonal}, and $\pi\colon M\otimes_k N\to M\otimes_\Lambda N$ is the canonical projection. It is straightforward to check $\Phi$ is a $\Lambda$-linear morphism of complexes that is compatible with the canonical augmentations to $M\otimes_\Lambda N.$ Thus $\Phi$ is a comparison map between semifree resolutions of $M\lotimes_\Lambda N$ over $\Lambda$, and so it is a homotopy equivalence. 

Applying $\Hom_\Lambda(-,k)$ to $\Phi$  yields the top map in the commutative diagram 
\[
 \begin{tikzcd}
 \Hom_\Lambda(\univ(M\otimes_\Lambda N),k) \ar[r,"\Phi^\vee"] \ar[d,swap,"\cong"] & 
 \Hom_\Lambda(\univ M\otimes_\Lambda \univ N,k) \ar[d,"\cong"] \\
 {\cob S}_{M\otimes_\Lambda N} \ar[r,"\Delta\otimes \pi^*"] & {\cob S}_{M}\otimes_k {\cob S}_{N} 
 \end{tikzcd}
\]
of dg $\cob S$-modules, where  ${\cob S}_{M}\otimes_k{\cob S}_{N}$ is viewed as a dg $\cob S$-module through $\Delta$.  The vertical parallel maps are isomorphisms by \cref{e:curveddg}; the one on the right also uses the standard isomorphisms
 \begin{align*}
 \Hom_\Lambda(\univ M\otimes_\Lambda \univ N,k)&\cong \Hom_\Lambda(\univ M,\Hom_\Lambda(\univ N,k))\\
 &\cong \Hom_\Lambda(\univ M,k)\otimes_k \Hom_\Lambda(\univ N,k) \,;
 \end{align*}
This is where the assumption that both $M$ and $N$ are degreewise finite rank and bounded below is needed. As $\Phi$ is a homotopy equivalence, from the commutativity of the diagram above it follows that $\Delta\otimes \pi^*$ is a homotopy equivalence of dg $\cob S$-modules justifying the first assertion; cf.\@ \cref{c:bgg}.
 
 With this in hand we have 
\[
 \h( \bgg(M\lotimes_\Lambda N))\cong \h( \bgg M \otimes_k \bgg N)\cong \h( \bgg M) \otimes_k \h( \bgg N)
 \]
 where the second map is the K\"unneth isomorphism. This gives the second of the following equalities
 \begin{align*}
 \V_\Lambda^\bgg(M\lotimes_\Lambda N)&=\supp_{\cob S} \h( \bgg(M\lotimes_\Lambda N))\\
 &=\supp_{\cob S} \left( \h( \bgg M) \otimes_k \h( \bgg N)\right) \\
 &=\join(\supp_{\cob S} \h( \bgg M), \supp_{\cob S} \h( \bgg N))\\
 &=\join(\V_\Lambda^\bgg(M),\V_\Lambda^\bgg(N))\,;
 \end{align*}
The third equality is \cref{l:support}.
 \end{proof}
 
\section{Passage to the exterior algebra}
\label{s:passage}

Throughout this section and the next $(R,\fm,k)$ is a commutative noetherian local ring. Fix a list of elements $\bmf=f_1,\ldots,f_c$ in $\fm$ and set 
\[
E\coloneq R\langle e_1,\ldots,e_c\mid\del e_i=f_i\rangle\,,
\] 
the Koszul complex on $\bmf$ over $R$, regarded as a local dg $R$-algebra in the standard way. One could take $R$ to be a local dg algebra where $\bmf$ is a list of cycles in even degrees, contained in the maximal ideal  of $R$; we stick to the situation above for ease of exposition. Two special cases are worth mention.

\begin{remark}
When $\bmf$ forms an $R$-regular sequence the augmentation $E\xra{\simeq}R/(\bmf)$ is a quasi-isomorphism of dg algebras and the map $R\to R/(\bmf)$ is complete intersection. When $\bmf$ is the zero sequence, $E$ is the exterior algebra over $R$ on $c$ generators of degree one. 
\end{remark}

Set $\Lambda \colonequals k\otimes_R E$ and $V\coloneq \Lambda_1$. We identify $e_1,\dots,e_c$ with their images in $\Lambda$; they are a basis for the $k$-space $V$. Set $W\coloneq \shift^{-1}(V^\vee)$,  and 
\[
\cob S\coloneq \Sym_k W\,.
\]  
Let $\chi_1,\dots,\chi_c$ be the basis of $W$ dual to $e_1,\dots,e_c$.

\begin{chunk}\label{c:univ2}
Let $M$ be a dg $E$-module whose homology is finitely generated over $R$. Let $F$ be a dg $E$-module that is semifree as a dg $R$-module, and $F\xra{\simeq} M$ an $E$-linear quasi-isomorphism. By \cite[Proposition~4.2.8]{Pollitz:2021}, $\RHom_E(M,k)$ can be equipped with a dg $\cob S$-module structure through the isomorphism 
\[
\RHom_E(M,k)\simeq \cob S\otimes_k \Hom_R(F,k)
\] 
where  the differential of the complex on the right is 
\[
1\otimes \del^{\Hom_E(F,k)}+\sum_{i=1}^c\chi_i\otimes \Hom(e_i,k)\,;
\] 
we let $\mathcal{C}_F$ denote this dg $\cob S$-module.  Following \cite[Definition~3.3.1]{Pollitz:2019}, the \emph{cohomological support of M over E} is 
\[
\V_E(M)=\supp_{\cob S} \Ext_E^*(M,k)=\supp_{\cob S}\h^*(\mathcal{C}_F)\,.
\]
A bridge to exterior algebras has been used effectively to acquire cohomological information on these support varieties when $R$ is regular and $\bmf$  is an $R$-regular sequence; see, for instance, \cite{Avramov/Iyengar:2010,Carlson/Iyengar:2015,Liu/Pollitz:2021}. This path is still sensible at this generality and can be used to establish results over $E$, as we do now.
\end{chunk}

Consider the functor $\mathsf{t}\colon \dcat{E}\to \dcat{\Lambda}$ given by $ k\lotimes_R -$. In the statement below, the construction of the dg $\cob S$-module ${\cob S}_{\mathsf{t} F}$ is given in \cref{c:bgg}

\begin{lemma}\label{l:passage}
Let $M$ be a dg $E$-module with finitely generated homology over $R$ and fix $F\xra{\simeq} M$ a quasi-isomorphism of dg $E$-modules where $F$ is semifree when regarded as dg $R$-module. There is the following isomorphism of dg $\cob S$-modules 
\[
\mathcal{C}_F\cong {\cob S}_{\mathsf{t} F}\,.
\]
In particular $\V_E(M)=\V_\Lambda^\bgg(\mathsf{t} M).$
\end{lemma}

\begin{proof}
For the isomorphism, since $\fm \Hom_R(F,k)=0$ the $E$-action on $\Hom_R(F,k)$ factors through $\Lambda$. It is immediate to check the adjunction isomorphism
\[
\alpha:\Hom_R(F,k)\xra{\cong} \Hom_k(\mathsf{t} F,k)
\] is one of $\Lambda$-modules. Therefore from the definitions of $\mathcal{C}_{F}$ and ${\cob S}_{\mathsf{t} F}$ in \cref{c:univ2} and \cref{c:bgg}, respectively, the map
\[
1\otimes \alpha: \mathcal{C}_{F}\to {\cob S}_{\mathsf{t} F}
\] 
is an isomorphism of dg $\cob S$-modules. The equality of supports  follows: 
\begin{align*}
 \V_E(M)&=\supp_{\cob S} \h(\mathcal{C}_F) \\
 &=\supp_{\cob S} \h({\cob S}_{\mathsf{t} F})\\
 &=\supp_{\cob S} \h({\cob S}_{\mathsf{t} M})\\
 &=\supp_{\cob S} \h(\bgg \mathsf{t} M)\\
 &=\V_\Lambda^\bgg(\mathsf{t}M)\,;
\end{align*} 
the second equality holds by the established isomorphism above and the others are clear from the various definitions. 
\end{proof}

\begin{remark} 
\label{c:cohsupp}Suppose $\bmf $ is an $R$-regular sequence and $M$ a finitely generated $R$-module such that $\bmf M=0$. The cohomological support of $M$ over $E$ agrees with support variety of $M$ introduced by Avramov in \cite{Avramov:1989}, and further developed in the work of Avramov and Buchweitz~\cite{Avramov/Buchweitz:2000b}.
 
 More generally, without the assumption $\bmf$ is regular, $\V_E(M)$ specializes to the support sets of Jorgensen \cite{Jorgensen:2002} and Avramov and Iyengar \cite{Avramov/Iyengar:2018}; cf.\@ \cite[Section 6.2]{Pollitz:2021}.  When $M$ has finite projective dimension over $R$, the cohomological support $\V_E(M)$ agrees with those above; hence \cref{l:passage} reveals how, in this setting, all of these supports are cohomological supports over an exterior algebra.
\end{remark}

\begin{chunk}
Let $\drel{E}{R}$ denote the full subcategory of $\dcat{E}$ consisting of objects $M$ such that $M$ is perfect when regarded as an object of $\dcat{R}$ via restriction of scalars. That is, if $\eta\colon R\to E$ is the structure map and $\eta_*\colon\dcat{E}\to \dcat{R}$ denotes the restriction of scalars functor along $\eta$, then $M$ is in $\drel{E}{R}$ if and only if $\eta_*(M)$ is isomorphic in $\dcat{R}$ to a bounded complex of finite rank free $R$-modules. In particular $M$ has bounded and finitely generated homology over $R$. When $R$ is regular $\drel{E}{R}$ is just the bounded derived category of dg $E$-modules. 
\end{chunk}

The result below is a particular case of a theorem of Gulliksen~\cite{Gulliksen/Levin:1969} and Avramov, Gashasrov, and Peeva~\cite{Avramov/Gasharov/Peeva:1997}.

\begin{proposition}
\label{c:fg}
For a dg $E$-module $M$, the following conditions are equivalent: 
\begin{enumerate}
 \item \label{c:fg1}$\Tor^R(k,M)$ is finitely generated over $k$; 
 \item \label{c:fg2}$\Ext_\Lambda(\mathsf{t} M,k)$ is finitely generated over $\cob S$;
 \item \label{c:fg3}$\Ext_\Lambda(k,\mathsf{t} M)$ is finitely generated over $\cob S$.
\end{enumerate}
Moreover, when $\h(M)$ is finite over $R$, the conditions above are equivalent to:
\begin{enumerate}\setcounter{enumi}{3}
 \item \label{c:fg4} $M$ is in $\drel{E}{R}$.
\end{enumerate}
\end{proposition}

\begin{proof}
The equivalence of \cref{c:fg1}, \cref{c:fg2}, and \cref{c:fg3} is from a special case of \cite[Theorem~4.3.2]{Pollitz:2021}, and the fact  \cref{c:fg1} and \cref{c:fg4} are equivalent when $\h(M)$ is finite over $R$, is classical; see, for example, \cite[Corollary~1.3.2]{Bruns/Herzog:1998}.
\end{proof}

\section{Support for tensor products, II}
\label{s:tp-II}
The notation in this section is as in the previous one. The result below is the theorem announced in the introduction. 

\begin{theorem}
\label{t:mainresult}
Suppose $E$ is a Koszul complex over a local ring $(R,\fm,k)$ on a finite list of elements  in $\fm.$ For $M,N$ in $\drel{E}{R}$, 
\[
\V_E(M\lotimes_E N)=\join(\V_E(M),\V_E(N))\,.
\]
\end{theorem}

\begin{proof}
 We need only pass to the exterior algebra: 
 \begin{align*}
 \V_E(M\lotimes_E N)&=\V_\Lambda^\bgg(\mathsf{t}(M\lotimes_E N))\\
 &=\V_\Lambda^\bgg(\mathsf{t}M\lotimes_\Lambda \mathsf{t} N)\\
 &=\join(\V_\Lambda^\bgg(\mathsf{t}M), \V_\Lambda^\bgg(\mathsf{t}N))\\
 &=\join(\V_E(M),\V_E(N))\,.
 \end{align*}
 The first and fourth equalities hold by \cref{l:passage}; the second one follows from the isomorphism 
 \[
 \mathsf{t}(M\lotimes_E N)\simeq \mathsf{t}M\lotimes_\Lambda \mathsf{t}N\,.
 \]
 By \cref{c:fg}, the dg $\Lambda$-modules $\mathsf{t}M,\mathsf{t}N$ are in $\dfbcat{\Lambda}$ and so the third equality holds by \cref{p:derivedtensor}.
\end{proof}

\begin{remark}
There is an alternative proof of \cref{t:mainresult}, using the Hopf algebra structure on $\Ext_E^*(k,k)$. The key point is that for any dg $E$-modules the maps
\begin{equation}
\label{e:multmap}
\Ext_E^*(M,k)\otimes_k \Ext_E^*(N,k)\to \Ext_E^*(M\lotimes_E N,k\lotimes_E k)\to \Ext_E^*(M\lotimes_E N,k)
\end{equation} 
are $\Ext_E^*(k,k)$-linear. The second map is induced by multiplication, $k\lotimes_E k\to k.$ In \cref{e:multmap} the graded Ext-module on the left is given an $\Ext_E^*(k,k)$-module structure through the diagonal
\[
\Ext_E^*(k,k)\to \Ext_E^*(k,k)\otimes_k\Ext_E^*(k,k)\,.
\]
This is a straightforward calculation. However this approach requires a bit of background on dg algebras with divided powers and suitably adapting classical material to this more general setting; cf.\@ \cite{Gulliksen/Levin:1969}. The main point is that $\Ext_E^*(k,k)$ is generated, as a $k$-algebra, by primitives induced by derivations that respect divided powers on the minimal semifree resolution of $k$ over $E$.

One can identify $\cob S$ as a Hopf subalgebra of $\Ext_E^*(k,k)$ so  the maps in \cref{e:multmap} are also $\cob S$-linear. When $M,N$ are in $\drel{E}{R}$, the $\cob S$-modules $\Ext_E^*(M,k),\Ext_E^*(N,k)$ are finite over $\cob S$, see \cref{c:fg}, so the assertion of \cref{t:mainresult} follows directly from \cref{l:support} once noting the composition in \cref{e:multmap} is an isomorphism.
\end{remark}

Consider the equivalence
 \[
 (-)^\dagger\colon \drel{E}{R}^{\textsf{op}}\longrightarrow \drel{E}{R}
 \] 
 where $M^\dagger\coloneq\RHom_E(M,E)$ for each $M$. 
 
 \begin{lemma}\label{l:dual}
 If $M$ is in $\drel{E}{R}$, then $ \V_E(M)=\V_E(M^\dagger).$
 \end{lemma}
 
 \begin{proof}
 As $M$ is perfect over $R$ there is an isomorphism of dg $\Lambda$-modules
 \[
 \mathsf{t} (M^\dagger)\simeq \RHom_\Lambda(\mathsf{t} M,\Lambda)\,.
 \] By \cite[Theorem~4.1]{Liu/Pollitz:2021}, $\mathsf{t}M$ and $\RHom_\Lambda(\mathsf{t} M,\Lambda) $ generate the same thick subcategory in $\dcat{\Lambda}.$ Thus the second equality below holds:
 \[
 \V_E(M)=\V_\Lambda^\bgg(\mathsf{t} M)=\V_\Lambda^\bgg(\RHom_\Lambda(\mathsf{t} M,\Lambda))=\V_\Lambda^\bgg(\mathsf{t}(M^\dagger))=\V_E(M^\dagger)\,.\qedhere
 \]
 \end{proof}
 
 \begin{corollary}
 \label{cor:rhom}
 If $R$ is Gorenstein and $\RHom_E(M,N)$ belongs to $\drel{E}{R}$, then 
 \[
 \V_E(\RHom_E(M,N))=\join(\V_E(M),\V_E(N)).
 \]
  \end{corollary}

 \begin{proof}
As $R$ is Gorenstein and the $R$-modules $\h(M),\h(N)$ are finitely generated,  there is an isomorphism
\[
\RHom_E(M,N)^\dagger\simeq M\lotimes_E N^\dagger\,.
\] 
Thus the second equality below holds
\begin{align*}
\V_E(\RHom_E(M,N))&=\V_E(\RHom_E(M,N)^\dagger) \\
&=\V_E( M\lotimes_E N^\dagger)\\
&=\join(\V_E(M),\V_E(N^\dagger))\\
&=\join(\V_E(M),\V_E(N))\,;
\end{align*}
the first and fourth equalities are by \cref{l:dual} and the third  is by \cref{t:mainresult}.
 \end{proof}
 
 \begin{remark}
 In light of \cref{t:mainresult}, it would be interesting to determine whether \cref{cor:rhom} holds without the assumption that $\RHom_E(M,N)$ is in $\drel{E}{R}$.
 \end{remark}
 
The result below relates the cohomological support of $M\lotimes_EN$ to those of its homology modules.  Specializing to the case $R$ is regular and $\bmf$ is an $R$-regular sequence, yields a positive answer to \cite[Question 2]{Dao/Sanders:2017}. The  containment in the statement of the theorem can be strict; see \cite[Example 5.3]{Dao/Sanders:2017}. 

\begin{theorem}
\label{t:second}
Let $M,N$ be in $\drel{E}{R}$ and suppose the $\cob S$-modules $\Ext_E^*(M,k)$ and $\Ext_E^*(N,k)$ are generated in cohomological degrees at most $s$ and $t$, respectively. There is a containment of closed subsets
	\[
	\V_E(M\lotimes_E N)\subseteq \bigcup_{i\leq s+t}\V_E(\Tor_i^E(M,N))\,.
	\]
\end{theorem}

\begin{proof}
By \cref{p:derivedtensor} and \cref{l:passage} one may identify  $ \Ext_E^*(M\lotimes_E N,k)$ with 
\[
\Ext_E^*(M,k)\otimes_k\Ext_E^*(N,k)
\]
viewed as a graded $\cob S$-module via restriction along the diagonal map \cref{polynomialdiagonal}. 
Let $T$ denote its graded $\cob S$-submodule of $\Ext_E^*(M\lotimes_E N,k)$ generated by 
\[
\bigoplus_{i+j\leqslant u}\Ext_E^i(M,k)\otimes_k \Ext_E^j(N,k)
\]
where $u=s+t$. The $(\cob S\otimes_k \cob S)$-module generated by $T$ is $ \Ext_E^*(M\lotimes_E N,k)$, so arguing as in the proof of \cref{l:support}, one gets an equality
\begin{equation}
\label{e:suppT}
\V_E(M\lotimes_E N)=\supp_{\cob S} T\,.
\end{equation}

Fix a semifree resolution $F\xra{\simeq} M\lotimes_E N$ over $E$, and let $F'$ be the soft truncation of $F$ in lower degrees at most $u$. Thus there is morphism of dg $E$-modules $\tau\colon F\to F'$ with the property that 
\[
\tau_{\leqslant u}\colon F_{\leqslant u}\to F'_{\leqslant u} 
\] 
is the identity map.  Hence
\[
\Ext(\tau,k)\colon \Ext_E^*(F',k)\to \Ext_E^*(M\lotimes_E N,k)
\] 
is an isomorphism in upper degrees at most $u$. In particular, under the identification discussed above
\[
T\subseteq \Image\left(\Ext_E^*(F',k)\xra{\Ext(\tau,k)} \Ext_E^*(M\lotimes_E N,k)\right)
\] 
and hence one has an inclusion
\[
\supp_{\cob S} T\subseteq \supp_{\cob S}\Ext_E^*(F',k)\,.
\] 
Since $F'$ has bounded homology, it is in the thick subcategory of $\dcat{E}$ generated by $\h(F')$ regarded as dg $E$-module via the augmentation $E\to \h_0(E)$. Thus
\[
\supp_{\cob S} T\subseteq \supp_{\cob S}\Ext_E\left(\oplus_{i\leqslant u} \Tor^E_i(M,N),k\right)=\bigcup_{i\leqslant u} \V_E(\Tor^E_i(M,N))
\]
 where for the first containment we are also using the equality 
 \[
 \h(F')=\bigoplus_{i\leqslant u} \Tor^E_i(M,N)\,.
\]
Combining this with \cref{e:suppT} finishes the proof. 
\end{proof}

\begin{remark}
\cref{t:second} implies that when $R$ is regular and  $M\lotimes_E N$ has finitely generated homology over $R$, the complexity of $M\lotimes_E N$, in the sense of \cite[Section~3]{Avramov:1989},  is bounded above by the maximum of the complexities of $\Tor^E_i(M,N)$ for $i\leq s+t$, where $s$ and $t$ are from \cref{t:second}.
\end{remark}


\bibliographystyle{amsplain}
\bibliography{refs}

\end{document}